\documentclass[pdflatex,sn-mathphys-num]{sn-jnl}% Math and Physical Sciences Numbered Reference Style
\usepackage{mathrsfs}
\usepackage{graphicx}%
\usepackage{multirow}%
\usepackage{amsmath,amssymb,amsfonts}%
\usepackage{amsthm}%
\usepackage{mathtools,thm-restate}
\usepackage{mathrsfs}%
\usepackage[title]{appendix}%
\usepackage{xcolor}%
\usepackage{textcomp}%
\usepackage{manyfoot}%
\usepackage{booktabs}%
\usepackage{algorithm}%
\usepackage{algorithmicx}%
\usepackage{algpseudocode}%
\usepackage{listings}%
\usepackage{tikz}
\usetikzlibrary{patterns.meta}
\usepackage{hyperref}
\usepackage[capitalise, nameinlink]{cleveref}

\theoremstyle{thmstyleone}%
\newtheorem{theorem}{Theorem}%  meant for continuous numbers
\newtheorem{proposition}[theorem]{Proposition}% 
\newtheorem{cor}{Corollary}

\theoremstyle{thmstyletwo}%

\theoremstyle{thmstylethree}%
\newtheorem{definition}{Definition}%
\newtheorem{conj}{Conjecture}
\newtheorem{question}{Question}
\newtheorem{obs}{Observation}

\Crefname{question}{Question}{Questions}

\newcommand{\p}{\mathcal P}
\newcommand{\n}{\mathcal N}
\newcommand{\sg}{\mathcal{G}}
\newcommand{\mex}{\rm{mex}}
\newcommand{\dps}{\mathrm{DPS}}
\newcommand{\ps}{\mathrm{PS}}
\newcommand{\s}{S}
\newcommand{\crim}{\textsc{CRIM} \/}

\newcommand{\cp}{\mathrm{CP}}

\newcommand{\OAE}{\mathrm{OAE}}
\newcommand{\EAO}{\mathrm{EAO}}
\newcommand{\y}{\mathbb{Y}}
\newcommand{\nim}{\textsc{Nim}}

\usepackage{array}

\newcommand{\drawpartition}[1]{
    \begin{tikzpicture}[scale=0.31, transform shape, baseline=(current bounding box.center)]
        \foreach \rowlength [count=\y] in {#1} {
            \ifnum\rowlength>0
                \foreach \x in {1,...,\rowlength} {
                    \draw[thick] (\x-1, 1-\y) rectangle ++(1, -1);
                }
            \fi
        }
    \end{tikzpicture}
}

\begin{document}

%\title[CRIM]{CRIM: An Impartial Combinatorial Game \\on Integer Partitions}
% \title[CRIM]{CRIM: An Impartial Combinatorial Game played by removing rows and columns of Integer Partitions}
\title[CRIM]{CRIM: A Natural Game on Integer Partitions}


\author[1]{\fnm{Ina} \sur{Bašić}}\email{ina.basic@tu-darmstadt.de}
\equalcont{These authors contributed equally to this work.}

\author[2]{\fnm{Eric} \sur{Gottlieb}}\email{gottlieb@rhodes.edu}
\equalcont{These authors contributed equally to this work.}

\author*[3]{\fnm{Matjaž} \sur{Krnc}}\email{matjaz.krnc@upr.si}
\equalcont{These authors contributed equally to this work.}

% Ina: fill this!
\affil[1]{\orgdiv{Department of Mathematics}, \orgname{Technical University of Darmstadt}, \orgaddress{\street{Karolinenplatz 5}, \city{Darmstadt}, \postcode{64289}, \country{Germany}}}

\affil[2]{\orgdiv{Department of Mathematics and Statistics}, \orgname{Rhodes College}, \orgaddress{\street{\\2000 North Parkway}, \city{Memphis}, \state{Tennessee}, \postcode{38112}, \country{USA}}}

\affil[3]{\orgdiv{FAMNIT}, \orgname{University of Primorska}, \orgaddress{\street{Glagoljaška 8}, \city{Koper}, \country{Slovenia}}}


\abstract{We analyze Column-Row Impartial Merge (CRIM), an impartial combinatorial game played on integer partitions. 
A move in CRIM consists of removing an arbitrary row or column from the corresponding Young diagram, with the remaining parts reattaching to form a single partition.
     We define \emph{rectairs} — a common generalization of rectangles and staircases — and characterize their $\mathcal{P}/\mathcal{N}$-status. 
     We define the \emph{meld} operation on partitions and show that the meld of losing rectairs is losing. 
  We introduce \emph{Odds-Are-Even} (OAE) and \emph{Evens-Are-Odd} (EAO) partitions, proving that all OAE partitions are $\mathcal{P}$-positions and characterizing the losing positions within EAO partitions. 
      We determine the $\mathcal{P}/\mathcal{N}$-status for staircases and for $2$- and $3$-part partitions. 
    We evaluate CRIM and its restrictions to certain partition families within the Conway-Gurvich-Ho classification scheme, establishing that CRIM is neither returnable nor domestic.
    We conjecture that every losing partition has even rank.}




\maketitle

\section{Introduction}
Combinatorial games have been studied extensively since the pioneering work of Bouton, Sprague, and Grundy \cite{bouton1901nim,Spr35,Spr37,Gru39}. In his book On Numbers and Games \cite{Con76}, Conway extended and generalized this work and placed it in a broader theoretical framework. 
These foundational results have shaped the study of impartial games, providing tools to determine whether a position is winning or losing (i.e. $\p/\n$ status), to decompose complex games, and to identify winning strategies. 
Some well-known combinatorial games (e.g. \textsc{Nim}, \textsc{Sato-Welter}, and \textsc{Wythoff}) admit interpretations in which positions are (integer) partitions. 

Partitions play important roles in number theory, combinatorics, representation theory, and mathematical physics 
\cite{bohr86,AgarwalaAuluck1951,sagan2001symmetric,AuluckKothari1946}
and have been studied by mathematicians including Euler \cite{euler1797introductio}, Erd\H os \cite{erdos1946some}, Hardy and Ramanujan \cite{hardy1918asymptotic}, Rademacher \cite{rademacher1938partition}, and Andrews \cite{andrews1998theory, andrews2018integer}. 
In 2018 and 2021, 
 Bloom and Saracino \cite{zbMATH06870132} and Bloom and McNew \cite{zbMATH07369446} studied column and row deletion of partitions in the context of rook polynomials, rook equivalence, and Wilf equivalence. 
In this article, we will study Column-Row Impartial Merge (\textsc{CRIM}), another combinatorial game on partitions, in which row and column deletion play an essential role.

\subsection{Related work} 
The study of games played on partitions bridges combinatorial game theory with the elegant theory of partitions, a cornerstone of number theory and enumerative combinatorics. 
The first such game that we know of was studied by Sato \cite{sato1970maya}, who showed that Welter's game can be formulated in this way and conjectured that its Sprague-Grundy values relate to the representation theory of the symmetric group. 
In 2018, Irie \cite{irie2018p} confirmed Sato's conjecture. 
Subsequent work by Abuku and Tada \cite{abuku2023multiple} and Motegi \cite{motegi2021gamepositionsmultiplehook} extended these results. 

Since then, a number of researchers have explored other combinatorial games on partitions: \textsc{LCTR} by several authors \cite{Ilic2019,Gottlieb2024LCTR}, \textsc{CRIM} by Ba\v si\' c \cite{basic2022some,Basic_2023}, \nim{} on partitions and hyperrectangles \cite{gottlieb2025nimintegerpartitionshyperrectangles}, and a suite of chess-inspired impartial games, collectively called \textsc{Impartial Chess}, by Berlekamp \cite{impchess} and others \cite{gottlieb2025impartialchessintegerpartitions}. 
In her honors thesis, Meit \cite{hannah2025thesis} studied \textsc{CRPM} and \textsc{CRPS}, two partizan combinatorial games on partitions.

Conway~\cite{Con76} introduced the notion of a tame game.  Gurvich and Ho~\cite{GURVICH201854} extended this idea to a more general classification of games. The Conway-Gurvich-Ho classification of the games \textsc{Nim}, \textsc{Subtraction}, \textsc{Wythoff}, \textsc{Mark}, \textsc{LCTR}, \textsc{PNim}, and \textsc{Downright} are known; see \cite{Con76,Gottlieb2024LCTR,gottlieb2025nimintegerpartitionshyperrectangles,GURVICH201854} and \cref{fig:cgh}.

\subsection{Our results}

This paper extends Bašić's work \cite{Basic_2023} on
\textsc{Column-Row Impartial Merge} $(\crim)$, a combinatorial impartial game played on Young diagrams where a move consists of removing an arbitrary row or column. If the removed row or column is neither the first nor the last, the remaining pieces of the Young diagram are reattached to each other, as illustrated in \cref{fig:cr-moves}.

We classify the $\p/\n$-status of partititions with two or three parts in \cref{thm:2parts,thm:3parts}. 
We introduce a new family of partitions, which we call \emph{hook-squares}, and classify them by $\p/\n$-status in \cref{prop:hooksquare}. 

Let $\lambda = [\lambda_1, \cdots, \lambda_r]$ be a partition and let $\lambda' = [\lambda_1', \cdots, \lambda_{\lambda_1}']$ be its conjugate. 
We say that $\lambda$ is an \emph{$\OAE$-partition} if $\lambda_i$ and $\lambda_i'$ are even whenever $i$ is odd. 
Similarly $\lambda$ is an \emph{$\EAO$-partition} if $\lambda_i$ and $\lambda_i'$ are odd whenever $i$ is even; see \cref{fig:even-odds}. 
\begin{figure}[h]
    \centering
\tikzset{every picture/.style={line width=0.75pt}} %set default line width to 0.75pt      
\includegraphics[scale=0.7]{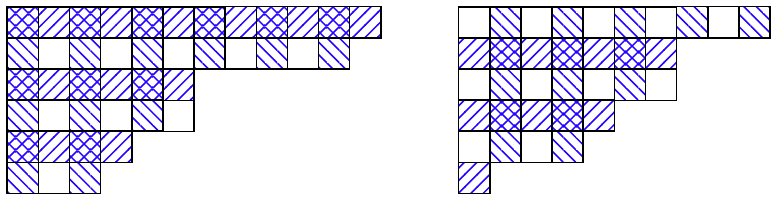}
    \caption{Example of an $\OAE$-partition (left) and an $\EAO$-partition (right). }
    \label{fig:even-odds}
\end{figure}

\noindent
The following theorem can be used to establish the $\p/\n$-status of a variety of families of partitions, including even staircases, partially resolving Conjecture 5.1 of Bašić \cite{Basic_2023}.

\begin{restatable}{theorem}{thmoddsareeven} \label{thm:oddsareeven}
$\OAE$-partitions are $\p$-positions.
\end{restatable}

\noindent   
The following characterization of the $\p$-positions among $\EAO$-partitions also partially resolves Conjecture 5.1 of Bašić \cite{Basic_2023}. 

\begin{restatable}{theorem}{corevensareodd} \label{cor:evensareodd}
Let $\lambda$ be an $\EAO$-partition with $\lambda_1 > 1$ columns and $r > 1$ rows. Then $\lambda$ is a $\p$-position if and only if it is of even rank. 
\end{restatable}

\noindent
We introduce a family of partitions which we refer to as \emph{rectairs} and the subfamily of \emph{balanced} rectairs, as well as an operation on partitions which we call \emph{meld} (see \cref{fig:rectairs}). 

\begin{figure}[h]
\centering
\includegraphics[scale=0.9]{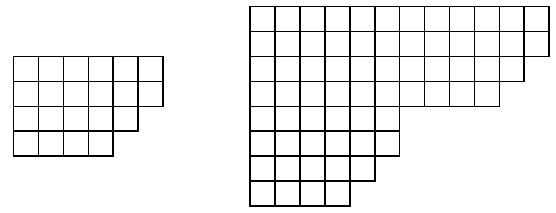}
\caption{The rectair $R^2_{4,6}$ and the meld $R^2_{4,6} \diamond R^2_{4,6}$.}
\label{fig:rectairs}
\end{figure}

\noindent
The meld operation is well-behaved with respect to $\p$-ness on balanced rectairs.

\begin{restatable}{theorem}{thmrectairs}
\label{thm:rectairs}
    The meld of any number of balanced rectairs is losing. 
\end{restatable}
\noindent From this, a number of other results follow, including (a second proof) that even staircases are losing %(\cref{proppnoffamilies})
and that the meld of even rectangles is losing.

We classify \crim played on several families of partitions according to the Conway-Gurvich-Ho hierarchy. %.

\begin{restatable}{theorem}{thmCGHmain}\label{thm:CGHmain}
    $\crim$  is neither returnable nor domestic. 
\end{restatable}

This paper is structured as follows. 
In \cref{sec:prelims}, we provide background information on combinatorial game theory, Sprague-Grundy theory, partitions, and $\crim$. 
In \cref{sec:CR}, we describe our results on $\crim$. 
In \cref{sec:CG}, we prove \cref{thm:CGHmain}, which describes the placement of \crim in the Conway-Gurvich-Ho (CGH) classification. 
We also determine the CGH classification of variants of CRIM played on restricted set of partitions, like hooks (\cref{thm:hooks}) and rectangles (\cref{thm:rects}). 
We conclude with \cref{sec:future}, where we offer some directions for future research. 

\section{Preliminaries} \label{sec:prelims}

In this section, we recall definitions and describe notations used throughout the paper. In \cref{subsec:pttns} we revisit the definitions of partitions and Young diagrams, as well as some related notions. In \cref{subsec:families} we describe some families of partitions that will be relevant to us. In \cref{subsec:game-theory} we review some basics of combinatorial game theory. 

\subsection{Partitions} \label{subsec:pttns}
Let $n$ be a non-negative integer. A way to write $n$ as a sum $\lambda_1 + \cdots + \lambda_r$ of positive integers with $\lambda_1 \geq \cdots \geq \lambda_r$ is called an \emph{(integer) partition}. The summands are called \emph{parts}. The set of partitions form a poset, known as Young's lattice \cite{Stanley}, which we denote by $\y$. 

We represent partitions graphically using Young diagrams\footnote{Young diagrams are sometimes refered to as Ferrers boards; see e.g. \cite{zbMATH06870132,zbMATH07369446}.} (see, e.g. \cref{fig:conjugation}).
If $\lambda = \lambda_1 + \cdots + \lambda_r$ is a partition we use tuple notation and write $\lambda = [\lambda_1, \ldots, \lambda_r]$. 
The \emph{(Dyson's) rank} of $\lambda$ is $|r - \lambda_1|$. 
The only partition of $0$ is the empty sum, denoted by $[ \, ]$. 
The \emph{conjugate} $\lambda'$ of $\lambda$ is a partition $\lambda' = [\lambda_1', \ldots, \lambda'_{\lambda_1}]$ where $\lambda'_k$ is the largest integer such that $\lambda_{\lambda'_k} \geq k$. 
Sometimes it's convenient to use exponential notation. 
For example, we write $[5^2, 4, 1^3]$ in place of $[5,5,4,1,1,1]$.

% \begin{figure}[h!]
%     \centering
%     \tikzset{every picture/.style={line width=0.75pt}} 
%     \includegraphics{figures/hookExample.pdf}
%     \caption{The hooks $\Gamma_{4,4}$, $\Gamma_{5,3}$ and $\Gamma_{6,2}$.}
%     \label{fig:hook-example}
% \end{figure}

Let $\mu = [\mu_1, \dots, \mu_k]$ and $\lambda = [\lambda_1, \dots, \lambda_m]$ be partitions. 
Define the \emph{meld} of $\mu$ and $\lambda$ to be the partition $\mu \diamond \lambda = [\mu_1 + \lambda_1, \dots, \mu_1 + \lambda_m, \mu_1, \mu_2, \dots, \mu_k].$ 
Observe that $\diamond$ is associative and that $[ \ ] \diamond \lambda = \lambda \diamond [ \ ] = \lambda$ for every partition $\lambda$. Thus, the set of partitions forms a (noncommutative) monoid under $\diamond$. 
% Every partition can be represented uniquely as a meld of rectangles. We refer to the rectangles in this meld as \emph{terminal rectangles}. 
Conjugation interacts nicely with meld: $(\mu \diamond \lambda)' = \lambda' \diamond \mu'$ (also see \cref{fig:conjugation}).

\begin{figure}[h!]
    \centering
   \includegraphics[scale=0.8]{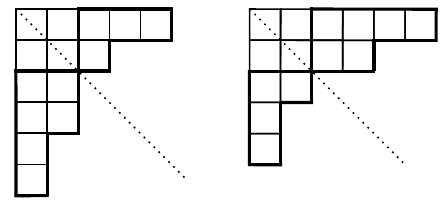}
    \caption{A Young diagram $\lambda = [5,3,2^2,1^2] = [2^2, 1^2] \diamond [3,1]$ and its conjugate $\lambda' = [6,4,2,1^2] = [2,1^2] \diamond [4,2]$.}
    \label{fig:conjugation}
\end{figure}

\newpage
\subsection{Families of Partitions} \label{subsec:families} 

Later in the paper, we obtain results for several families of partitions. We now define these families for future reference.  Let $r, c \geq 1$. See \cref{fig:all-partitions}. 
    \begin{itemize}
        \item \label{def:rowcol}
     A \emph{row} is a partition of the form $[c]$. A \emph{column} is a partition of the form $[1^r]$. \label{def:hook}
   \item A \emph{hook} is a partition of the form $[c, 1^{r-1}]$, which we denote by $\Gamma_{r, c}$. 
     A \emph{thick hook}, denoted $\Gamma_{r,c}^{a, b}$, is a partition of the form $[c^a,b^{r-a}]$, where 
     $r > a \geq 2$ and $c > b \geq 2$. 

\item \label{def:rectangle}
   A \emph{rectangle} $R_{r,c}$ is a partition of the form $[c^r]$.
    We include the empty partition $R_{0,0}=[\ ]$, and call it the \emph{trivial} rectangle.
 
    \item 
     \label{def:quadrated}
      A \emph{quadrated} partition is one in which each part is even and appears an even number of times, i.e., one of the form $[\lambda_1^{m_1}, \ldots, \lambda_t^{m_t}]$, where each $\lambda_j$ and $m_j$ is even. 
  % For $n \geq 1$, a \emph{staircase}, denoted $S_n$, is a partition of the form $[n, n-1, \ldots, 1]$. For $n \geq 2$, a \emph{padded staircase}, $PS_n$ is a partition of the form $[n, n, n-1, \dots, 1]$. For $n \geq 2$, a \emph{doubly padded staircase}, $DPS_n$ is a partition of the form $[n+1, n+1, n, \ldots, 2]$. 
  \item \label{def:staircase}
  For $n \geq 1$, a \emph{staircase}, denoted $\s_n$, is a partition of the form $[n, n-1, \ldots, 1]$. For $n \geq 2$, a \emph{padded staircase}, $\ps_n$ is a partition of the form $[n, n, n-1, \dots, 1]$. For $n \geq 2$, a \emph{doubly padded staircase}, $\dps_n$ is a partition of the form $[n+1, n+1, n, \ldots, 2]$. See \cref{fig:all-partitions}.   
  \item Let $\lambda = [\lambda_1, \cdots, \lambda_r]$ and let $\lambda' = [\lambda_1', \cdots, \lambda_{r'}']$ be its conjugate. We say that $\lambda$ is an \emph{$\OAE$-partition} if $\lambda_i$ and $\lambda_i'$ are even whenever $i$ is odd. Similarly $\lambda$ is an \emph{$\EAO$-partition} if $\lambda_i$ and $\lambda_i'$ are odd whenever $i$ is even. We denote the set of $\OAE$-partitions by $\OAE$ and the set of $\EAO$-partitions by $\EAO$; see \cref{fig:even-odds}.
  \item Let $r, c$ be positive integers and let $0 \leq k < \min(r, c)$. A \emph{rectair} is a partition that is either $[ \ ]$ or of the form $R^k_{r, c} = [c^{r-k}, c-1, \ldots, c-k]$.

    \end{itemize}
    
    Observe that rectairs are a common generalization of rectangles, stairs, padded stairs, and doubly padded stairs. Indeed, for 
suitable values of $r, c, k$, we have \[R^0_{r, c} = R_{r, c}, \quad R^{k-1}_{k, k} = \s_k, \quad R^{k-1}_{k+1, k} = \ps_{k}, \quad \mbox{and} \quad R^{k-2}_{k, k} = \dps_{k}.
\] 

\begin{figure}[h!]
    \centering
    \includegraphics[scale=.97]{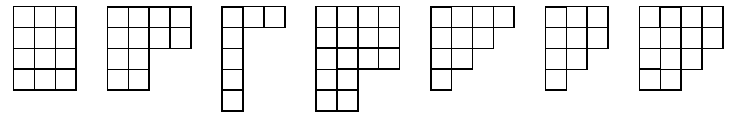}
    \caption{
    The rectangle $R_{4,3}$, 
    the quadrated partition $[4^2, 2^2]$, 
    the hook $\Gamma_{5,3}$, 
    the  thick hook $\Gamma_{5,4}^{3,2}$, 
    the staircase $S_4$, 
    the padded staircase $\ps_3$ and 
    the doubly padded staircase $\dps_3$. 
    We note that $\Gamma_{5,3}$ and $S_4$ are EAO partitions; that $[4^2, 2^2]$ is OAE; and that $S_4,\ps_3,$ and $\dps_3$ are rectairs.}
    \label{fig:all-partitions}
\end{figure}

% \begin{definition}
%     A partition $\lambda = [\lambda_1, \ldots, \lambda_r]$ is said to be \emph{stable} if $\lambda_1 > 1$ and $r>1$ and $\lambda_1$ and $r$ have the same parity. 
% \end{definition}

\subsection{Combinatorial games and their classification \label{subsec:game-theory}}
% we decided not to comment on distinguishing between position and game
In this paper we consider finite (short) impartial games, that is, players have the same moves available to them throughout a game, and after finitely many steps the games concludes with a uniquely defined winner. 
Throughout the paper all positions in the games are partitions.

A \emph{winning position}, also known as an $\n$-position, is a position in which the next player can guarantee a win. 
Conversely, a \emph{losing position}, equivalently a $\p$-position, is a position from which the previous player can guarantee a win. Whenever a move leads to a $\p$ position, we also say that such a move \emph{wins.}
%Notice that this implies that there is \emph{at least one} losing position which can be reached from an $\mathcal{N}$-position while \emph{every} position reachable from a $\p$ position has to be a winning one.
We call a position \emph{terminal} if no further moves can be made from it.

There are two ways to play a combinatorial game, depending on the winning condition. In \emph{normal play}, a player who reaches a terminal position loses. In \emph{misère play}, a player who reaches a terminal position wins. 

%Alternatively, the terminal position under normal play is always a $\mathcal{P}$-position.
%For simplicity, unless otherwise specified, all games outlined in further text are impartial and normal play. A \emph{solved} game is a game whose outcome can be efficiently calculated given any one of its positions.

Given an impartial game $G$, let  $\lambda$ be one of its positions.
For normal-play impartial games, Sprague \cite{Spr35} and Grundy \cite{Gru39} independently found a way to generalize winning and losing positions, assigning to each position in an impartial combinatorial game a nonnegative integer value, called the \emph{Sprague-Grundy value} of the position. We will denote the Sprague-Grundy value of a position $\lambda$ by $\sg_G(\lambda)$
and the related \emph{minimal excluded value} of a finite set of nonnegative integers $S$ by $\mathrm{mex}(S)$. For an in-depth treatment of the concepts mentioned in this subsection we refer the reader to \cite{Sie13}.

% We use standard notation (see \cite{Con76}\todo{add}).
The \emph{mis\`ere Grundy numbers} $\mathcal G_G^-$ are defined the same way as the ordinary Sprague-Grundy values, where the recursion end with terminal positions having the mis\`ere Grundy value equal to one.
Let $\sg_G(\lambda)$ and $\mathcal G_G^-(\lambda)$ be its Sprague-Grundy and mis\`ere Grundy values, respectively.
Then, we define  the \emph{Conway pair} of $\lambda$ to be $\cp_G(\lambda) = (\sg_G(\lambda), \mathcal G_G^-(\lambda))$.
We omit the subscript when $G$ is clear from context.

\begin{definition} \label{D.DTP}
Recall that a game is said to be \begin{enumerate} 
\item {\em returnable} if for any move from a $(0,1)$-position (resp.,~a $(1,0)$-position) to a non-terminal position $y$, there is a move from $y$ to a $(0,1)$-position (resp.,~to a $(1,0)$-position).
\item {\em forced} if each move from a $(0,1)$-position results in a $(1,0)$-position and vice versa;
\item {\em domestic} if it has neither $(0,k)$-positions nor $(k,0)$-positions with $k \geq 2$;
\item {\em tame} if it has only $(0,1)$-positions, $(1,0)$-positions, and $(k,k)$-positions with $k \geq 0$;
\item  {\em miserable} if for every position $x$, one of the following holds: (i) $x$ is a $(0, 1)$-position or $(1,0)$-position, or (ii) there is no move from $x$ to a $(0,1)$-position or $(1, 0)$-position, or (iii) there are moves from $x$ to both a $(0, 1)$-position and a $(1, 0)$-position.
\item {\em pet} if it has only $(0,1)$-positions, $(1,0)$-positions, and $(k,k)$-positions with $k \geq 2$.
 \end{enumerate}
\end{definition}
 If $\cp(\lambda) \in \{(0, 1), (1, 0)\}$ then $\lambda$ is called \emph{swap}. If $k \geq 0$ and $\cp(\lambda) = (k, k)$, then $\lambda$ is called \emph{symmetric}.

\subsection{Previously established results about \crim} \label{sec:SG}
In \cite{basic2022some,Basic_2023,Ilic2019,GKM22LCTR}, CRIM and the related game LCTR\footnote{Note that in the case of rectangular partitions, the games CRIM and LCTR coincide.} are studied and for several partition families the Grundy values are explicitly determined. 
%Some of the following results will be useful in subsequent sections. 

\begin{proposition}
\label{lem:bordercase_a2}
    Let $3\le \lambda_2\le \lambda_1$ be positive integers.
    Then 
    $$\sg \left([\lambda_1,\lambda_2]\right) = 
\begin{cases}
0 &  \text{if $\lambda_1$ is even; }\\
2 & \text{if $\lambda_1$ and $\lambda_2$ are odd;} \\
3 & \text{otherwise.}
\end{cases}$$
\end{proposition}

\begin{proposition}
\label{lem:rectangle}
For positive integers $a$ and $b$, we have
$$\sg(R_{r,c}) = 
\begin{cases}
0 & \text{if } c > 1 \text{ and } r > 1 \text{ and } r+c \text{ is even; }\\
2 & \text{if } c \leq 2 \text{ or } r \leq 2 \text{ and } r+c \text{ is odd; } \\
1 & \text{otherwise.}
\end{cases}$$
\end{proposition}

\begin{proposition}
\label{lem:hook}
The Sprague-Grundy value of a hook partition $\Gamma_{r,c}$ is given by:
$$\sg(\Gamma_{r,c}) = 
\begin{cases}
(r+c) \bmod 2 + 1 & \text{if } \min(r,c) = 1; \\
0 & \text{if $r$ and $c$ have the same parity;} \\
3 & \text{otherwise.}
\end{cases}$$ 
\end{proposition}
Let $c$, $b$, and $r$ be positive integers such that $c \geq b$. An \emph{almost-hook}  $\Gamma_{r,c}^{1,b}$  is a partition of form $[c,b^{r-1}]$, with $r$ rows and $c$ columns such that the first column is repeated $b$ times.

\begin{theorem}
\label{thm:almost-hook}  
Let $\Gamma_{r,c}^{1,b}$ be an almost-hook such that 
$r \geq 4, \  c \geq 3, \ b \geq 3$ and $b \leq c$. 
Then 
$$\sg \left(\Gamma_{r,c}^{1,b} \right) = 
\begin{cases}
(r+c) \bmod 2 & \text{if $c$ and $b$ have the same parity;}\\
2 & \text{if $r$ and $c$ have the same parity distinct from $b$;} \\
3 & \text{otherwise.}
\end{cases}$$
\end{theorem}

\begin{theorem}
\label{thm:thick-hook}
Let $\sg\left(\Gamma_{r,c}^{a,b}\right)$ be a thick hook. Then
$$\sg\left(\Gamma_{r,c}^{a,b}\right) = 
\begin{cases}
0 & \text{if $r$ and $c$ have the same parity} \\
2 & \text{if $r$ and $c$ have different parities and $a = b = 2$} \\
1 & \text{if $r$ and $c$ have different parities} 

\end{cases}$$
\end{theorem}

% The following corollary follows immediately from ***. We conjecture that it is true more generally; see ***. 
% \begin{cor}
%     Let $\lambda$ be a rectangle, hook, almost hook, or thick hook. Then $\lambda \in \p$ if and only if $\lambda$ is stable. 
% \end{cor}
\newpage
\section{\crim}  \label{sec:CR}

\crim is played on partitions. If $\lambda = [\lambda_1, \cdots, \lambda_r]$ then a  \emph{row move} consists of removing the $i$th part of $\lambda$ for some $i$, $1 \leq i \leq r$, yielding the partition $[\lambda_1, \ldots, \lambda_{i-1}, \lambda_{i+1}, \ldots, \lambda_r]$. A \emph{column move} consists of taking the conjugate of the partition that results from making a row move on $\lambda'$. We can think of \crim as being played on Young diagrams, where a row (resp. column) move consists of removing a row (resp. column) of the Young diagram (see \cref{fig:cr-moves}).

\begin{figure}[h!]
    \centering
\tikzset{every picture/.style={line width=0.75pt}}     
\includegraphics[scale=0.85]{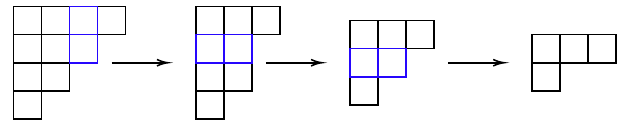}
\caption{ An example of a game of \crim. The second move is the perpendicular response to the removal of the third column, while the third one is the parallel response to the removal of the second row.}
\label{fig:cr-moves}
\end{figure}

\noindent
The following two responses are often useful in describing winning strategies for the previous player, when they exist. 

\begin{definition}
    Let $\lambda = [\lambda_1,\dots, \lambda_r]$ be a partition with $\lambda_1,r\ge 2$ and let $\lambda' = [\lambda'_1,\dots, \lambda'_{\lambda_1}]$ be its conjugate. Then: 
    \begin{itemize}
        \item For any $i\in [r]$ we say that the \emph{perpendicular response} to the removal of the $i$th row of $\lambda$ is the removal of the $\lambda_i$-th column of the resulting partition.
        \item For any $i\in \left[2\left\lfloor r/2\right\rfloor \right]$ we say that the \emph{parallel response} to the removal of the $i$th row of $\lambda$ is the removal of the $i$th row of the resulting partition when $i$ is odd, and the $(i-1)$th row otherwise.
        \item For any $i\in [\lambda_1]$ we say that the \emph{perpendicular response} to the removal of the $i$th column of $\lambda$ is the removal of the $\lambda'_i$-th column of the resulting partition.
        \item For any $i\in \left[2\left\lfloor \lambda_1/2\right\rfloor \right]$ we say that the \emph{parallel response} to the removal of the $i$th column of $\lambda$ is the removal of the $i$th column of the resulting partition when $i$ is odd, and the $(i-1)$th column otherwise.
    \end{itemize}
\end{definition}

\subsection{Partitions with a small number of parts} 

We now characterize losing partitions having two or three parts. An alternative proof of \cref{thm:2parts} for $3 \leq \lambda_2 \leq \lambda_1$ follows from \cite[Lemma 3.8]{Basic_2023}. \cref{thm:2parts} is a stronger version of this result since it only requires that $1 \leq \lambda_2 \leq \lambda_1$. 

\begin{restatable}{proposition}{thmtwoparts} \label{thm:2parts}
    A partition $\lambda = [ \lambda_1, \lambda_2] \in \p$ if and only if $\lambda_1$ is even and $\lambda_2 > 0$. 
\end{restatable}

\begin{proof}
    We proceed by induction on $\lambda_1$. First suppose $\lambda_1$ is even. 
    For the base step, recall that $[2,1] \in \p$ and $[2,2] \in \p$. 
    For the induction step, suppose $\lambda_1 > 2$. Note that either row move allows for a winning parallel response, so suppose the first player removes a column. 
    If the column move intersects only the first row, then the removal of rightmost column is a winning response. 
    If the column move intersects the second row and $\lambda_2 = 1$ then removing the sole remaining row wins. 
    Finally, if the column move intersects the second row and $\lambda_2>1$ then any column move is a winning response. 

    Now suppose $\lambda \in \p$. 
    If $\lambda_1$ is odd and at least 3, then the removal of a column other than the first one leaves a partition $\bar \lambda = [\bar \lambda_1, \bar \lambda_2]$ where $\bar \lambda_1$ is even. 
    By induction, $\bar \lambda \in \p$, a contradiction. So $\lambda_1$ is even. 
\end{proof}

% \begin{theorem}
%     If $\lambda = [\lambda_1, \lambda_2, \lambda_3]$ then $\lambda \in \p$ if and only if $\lambda_1$ and $\lambda_2$ are odd and either $\lambda_2 = \lambda_3 = 1$ or $\lambda_3 > 1$. 
% \end{theorem}
\begin{restatable}{theorem}{thmparts} \label{thm:3parts}
    Let $\lambda = [\lambda_1, \lambda_2, \lambda_3]$ be a partition on three parts. 
    Then $\lambda \in \p$ if and only if $\lambda_1$ and $\lambda_2$ are odd and either $\lambda$ is a hook or $\lambda_3 > 1$. 
\end{restatable}
% \begin{theorem} \label{thm:3parts}
%     Let $\lambda = [\lambda_1, \lambda_2, \lambda_3]$ be a partition on three parts. 
%     Then $\lambda \in \p$ if and only if $\lambda_1$ and $\lambda_2$ are odd and either $\lambda$ is a hook or $\lambda_3 > 1$. 
% \end{theorem}
\begin{proof}
    Suppose $\lambda \in \p$. If $\lambda_j$ is even for $j = 1$ or $j = 2$, then the removal of $\lambda_{2-j}$ wins by \cref{thm:2parts}, a contradiction. Thus, $\lambda_1$ and $\lambda_2$ must both be odd. If $\lambda_2 > 1$ and $\lambda_3 = 1$, then the removal of the first column wins, again by \cref{thm:2parts}, another contradiction, so either $\lambda$ is a hook or $\lambda_3 > 1$. 

    Now suppose that $\lambda_1$ and $\lambda_2$ are both odd. If $\lambda$ is a hook then $\lambda \in \p$ by \cref{lem:hook}. Suppose instead that $\lambda_3 > 1$. We show that every move has a winning response. If the first player removes a column, then the removal of the third row wins by \cref{thm:2parts}. If the first player removes a row, then the removal of the first column wins, again by \cref{thm:2parts}. Thus $\lambda \in \p$. 
\end{proof}

\cref{prop:hooksquare} is closely related to \cref{lem:hook}.

\begin{restatable}{proposition}{prophooksquare} \label{prop:hooksquare}
    Let $r \ge 2$ and $c \ge 2$ be integers.
    Then $[c,2,1^{r-2}] \in \p$ if and only if $r$ and $c$ are even.
\end{restatable}
\begin{proof}
    Let $\lambda = [c,2,1^{r-2}]$ and suppose $r$ and $c$ are even. 
    By symmetry it is enough to show that for any $i\in \{1,2, \dots,r\}$ the removal of $i$th row, yielding $\bar \lambda$, admits a winning response.
    
    Suppose $i\in\{1, 2\}$. If $r > 2$, the removal of the second row of $\bar \lambda$ wins by \cref{lem:hook}.  If $r = 2$, the parallel response wins.
    Similarly, if $i\ge 3$, the parallel response wins. 

    For the other direction, suppose without loss of generality that $c$ is odd. We will show that there is a move to a $\p$-position. If $r$ is even, then the removal of the third column wins. If $r$ is odd, the removal of the first column wins. 
\end{proof}

\subsection{Odds-Are-Even and Evens-Are-Odd partitions}

In this section we prove the following theorem, and then use it to establish the $\n/\p$-status of a variety of families of partitions.
\thmoddsareeven*
\begin{proof}
Let $\lambda \vdash n$ be an $\OAE$-partition. We proceed by induction on $n$.
The result holds if $n=0$ so suppose $n > 0$.
Regardless of the move of the first player, the previous player can make a parallel response, leading to partition $\bar{\bar \lambda}$.
This is always possible because the number of rows is even. 

It remains to show that $\bar{\bar \lambda}$ is an OAE-partition.
Without loss of generality assume that the rows $k-1$ and $k$ are removed and let $\bar{\bar \lambda} = [\bar{\bar \lambda}_1, \ldots, \bar{\bar \lambda}_{r-2}] = [\lambda_1, \ldots, \lambda_{2k-2}, \lambda_{2k+1}, \ldots, \lambda_{r}]$ be the resulting partition. 
As the lengths of rows did not change, the odd-indexed rows of $\bar{\bar \lambda}$ remain of even length. It remains to show that the odd-indexed columns of $\bar{\bar \lambda}$ shortened by either $0$ or $2$. 
Consider the $j$th column, where $j$ is odd. Either $\lambda_{2k-1} < j$, in which case $\bar{\bar \lambda}_j' = \lambda_j'$, or $\lambda_{2k-1} \ge j$. If $\lambda_{2k} < j$ then $\lambda_j'$ would be odd, contrary to assumption. 
Thus, we must also have $\lambda_{2k} \geq j$ since $\lambda_j'$, so $\bar{\bar \lambda}_j' = \lambda_j'-2$. 
\end{proof}

The sequence with $n$th term equal to the number of odds-are-even partitions among the partitions of $n$ is not recognized by OEIS. %0, 0, 1, 1, 0, 0, 2, 2, 0, 1, 4, 3, 0, 2, 7, 5, 0, 5, 12, 7, 1, 9, 19, 11, 2, 17, 30, 15, 5, 28

Next we use \cref{thm:oddsareeven} to characterize losing $\EAO$-partitions. 

\begin{proposition} \label{prop:evensareodd1}
Let $\lambda$ be an $\EAO$-partition with $\lambda_1 > 1$ columns and $r > 1$ rows. If $\lambda$ has even rank, then $\lambda$ is a $\p$-position. 
\end{proposition}

\begin{proof}
Let $\lambda = [\lambda_1, \ldots, \lambda_r]$. Without loss of generality assume that the first player takes a row. We use induction on $r$ to show that $\lambda \in \p$. 
We treat the base cases $r = 2$ and $r = 3$. 
If $r = 2$ then $\lambda_1$ is even since the rank is even so $\lambda \in \p$ by \cref{thm:2parts}. 
If $r = 3$ then $\lambda_2$ is odd. Since there cannot be two columns of length $2$, we have $\lambda_2-\lambda_3\le1$ so $\lambda \in \p$ by \cref{thm:3parts}.

For the induction step suppose $r\ge 4$. We distinguish two cases. If the first player takes the the first row, the previous player should take the first column. The resulting partition is in $\OAE$ and hence in $\p$ due to \cref{thm:oddsareeven}.
Assume now that the first player takes the $i$th row with $i\ge 2$, resulting in the partition $\bar \lambda$. The previous player should remove the $(i-1)$st row of $\bar\lambda$ if $i$ is odd, and the $i$th otherwise. As in the proof of \cref{thm:oddsareeven}, the resulting partition is in $\EAO$. 
\end{proof}

We now prove the converse of \cref{prop:evensareodd1}.

\begin{restatable}{proposition}{propevensareoddd} \label{prop:evensareodd2}
Let $\lambda$ be an $\EAO$-partition with $\lambda_1 > 1$ columns and $r > 1$ rows. If $\lambda$ has odd rank, then $\lambda$ is an $\n$-position. 
\end{restatable}
% \propevensareoddd*
\begin{proof}
Let $\lambda = [\lambda_1, \ldots, \lambda_r]$, and let  $\lambda' = [\lambda'_1, \ldots, \lambda'_{\lambda_1}]$ be its conjugate.
We will show that a $\p$ position is reachable from $\lambda$ in a single move.
Observe that
$(\lambda_1 - \lambda_2) + (\lambda_1' - \lambda_2') = (\lambda_1 + r) - (\lambda_2 + \lambda_2')$ is odd, so one of $\lambda_1 - \lambda_2$ and $\lambda_1' - \lambda_2'$ is positive. Thus, there exists a row or column of length one in $\lambda$. If the first player removes it, the resulting partition is losing due to \cref{prop:evensareodd1}. 
\end{proof}

Combining \cref{prop:evensareodd1,prop:evensareodd2}, we obtain the following. 
\corevensareodd*

Recall the definitions of staircase, padded staircase, doubly padded staircase, and quadrated partitions from \cref{subsec:families}. Next we use \cref{thm:oddsareeven,cor:evensareodd} to show establish the $\p/\n$-status of each of these families of partitions.

% Changed observation 1 to proposition 1 and removed corollary about even staircases being in P.
\begin{restatable}{proposition}{proppnoffamilies} \label{prop:pn-of-families}   
    The staircase $\s_n$ is in $\p$ if and only if $n$ is even.
    Padded staircases are in $\n$.
    Doubly padded staircases and quadrated partitions are in $\p$.
\end{restatable}

\begin{proof}
We prove each of the claims, treating the parities separately for doubly padded staircases.
\begin{description}
    \item[\textbf{Staircases:}]
    \cref{thm:oddsareeven} (or \cref{cor:evensareodd}) implies $\s_n \in \p$ if $n$ is even. If $n$ is odd then $\s_n \in \n$ as we can reach the $\p$-position $\s_{n-1}$ in a single move.
    
    \item[\textbf{Odd doubly padded staircases:}] \cref{cor:evensareodd} implies $\dps_n \in \p$ for odd $n$. 
    
    \item[\textbf{Padded staircases:}] If $n$ is even, the even staircase $S_n$ can be reached in a single move. Notice that $\ps_1 = [1,1]$ which is in $\n$. If $n > 1$ is odd, the odd doubly padded staircase $\dps_n$ can be reached in a single move. 

    \item[\textbf{Even doubly padded staircases:}] 
    For a positive even integer $n$, whatever move the first player makes on $\dps_{n}$, the perpendicular response leads to an odd doubly padded staircase $\dps_{n-1}$.

\item[\textbf{Quadrated partitions:}] 
The result for the quadrated partitions follow from \cref{thm:oddsareeven} due to all (odd) columns and rows being of even length.
\qedhere
\end{description}
\end{proof}

% \begin{definition}
% Let $r, c$ be positive integers and let $0 \leq k < \min(r, c)$. A \emph{rectair} is either $[]$, or a partition $R^k_{r, c} = [c^{r-k}, c-1, \ldots, c-k]$.

% We say that a rectair $R^k_{r, c}$ is \emph{balanced} if it has even rank, $r$ and $c$ are at least $2$, and, when $r$ and $c$ are odd, $k < \min(r, c) -1$. The empty rectair $[ \, ]$ is balanced. 
% \end{definition}

% Rectairs generalize several families of partitions. 
% For example, for suitable values of $r, c, k$, we have \[R^0_{r, c} = R_{r, c}, \quad 
% R^{k-1}_{k, k} = \s_k, \quad 
% R^{k-1}_{k+1, k} = \ps_{k}, \quad 
% \mbox{and} \quad 
% R^{k-2}_{k, k} = \dps_{k}.
% \] 

% In each case above, the rectair is balanced if and only if it is in $\p$ by \cref{lem:rectangle,prop:pn-of-families}.
% Similar is true for any rectair $R^k_{r, c}$ with $k$ odd and both $r,c$ even (\cref{thm:oddsareeven}). 
% Next we show that the property of being balanced characterizes the rectairs in $\p$.
% We will show that a rectair is balanced if and only if it is losing.
% We show this by showing the following stronger result.
A rectair $R^k_{r, c}$ is in $\OAE$ (and is therefore in $\p$ due to \cref{thm:oddsareeven}) if and only if $r$ and $c$ are even while $k$ is odd. Similarly, in each of the familiar cases above, the partition being in $\p$ is inherently tied to the parity of $r,c$ and $k$ in some way. In the next section we give necessary and sufficient conditions for $R^k_{r, c}$ to be in $\p$. 

\subsection{Melds of balanced rectairs} \label{sec:melds}

Every partition can be represented uniquely as a meld of rectangles. We refer to the rectangles in this meld as \emph{terminal rectangles}. 

\begin{theorem} \label{thm:ltrl}
    The meld of losing rectangles is losing.
\end{theorem}

\begin{proof}
We proceed by induction on the size of the partition $\lambda$. 
The result is true if $k \le 1$ (see \cref{lem:rectangle}) so suppose $k > 1$, with
 $\lambda = R_{r_1, c_1} \diamond \cdots \diamond R_{r_k, c_k}$ where $r_j + c_j$ is even and $r_j, c_j \geq 2$ for for $j = 1, \cdots, k$. 

Without loss of generality suppose the first player removes a row, yielding a partition $\bar \lambda$. Then $\bar \lambda = R_{r_1, c_1} \diamond \cdots \diamond R_{r_j-1, c_j} \diamond \cdots \diamond R_{r_k, c_k}$ for some $j$. In other words, a row move reduces exactly one rectangle in the meld. 
We will show that if $r_j$ is odd, then the perpendicular response is winning, while otherwise the parallel response is winning.

If $r_j$ is odd (and therefore $r_j\ge 3$) then the perpendicular response removes one row and one column from the $j$th rectangle giving a partition $\bar{\bar \lambda}$ where the $j$th rectangle in the meld for $\lambda$ is replaced with one having both dimensions even. If $r_j$ is even the situation is slightly more subtle. After the parallel response we obtain
\begin{align*}
    \bar {\bar \lambda} &= R_{r_1, c_1} \diamond \cdots \diamond R_{r_j - 2, c_j} \diamond \cdots \diamond R_{r_k, c_k} &  \quad \text{ if }& r_j > 2 \\
    \bar {\bar \lambda} &= R_{r_1, c_1} \diamond \cdots \diamond R_{r_{j-1}, c_{j-1}}
     \diamond R_{r_{j+1}, c_{j+1}+c_j} \diamond \cdots \diamond R_{r_k, c_k} & \quad \text{ if }& r_j = 2 \text{ and } j < k \\
     \bar {\bar \lambda} &= R_{r_1, c_1} \diamond \cdots \diamond R_{r_{k-1}, c_{k-1}} & \quad \text{ if }& r_j = 2 \text{ and } j = k.
\end{align*}
% \begin{enumerate}
%     \item \label{it:1}If $r_j = 2$ and $j = 1$, removing the remaining row of $j$th rectangle gives
%     \[\bar {\bar \lambda} = R_{r_2, c_2} \diamond \cdots \diamond R_{c_k, r_k}.\]
%     \item If $r_j = 2$ and $j > 1$, the same response as in \cref{it:1} gives 
%     \[\bar {\bar \lambda} = R_{r_1, c_1} \diamond \cdots \diamond R_{r_{j-1}, c_{j-1} + c_j} \diamond R_{r_{j+1}, c_{j+1}} \diamond \cdots \diamond R_{c_k, r_k}.\]
%     \item If $r_j = 3$ then removing a column that reduces $R_{r_j, c_j}$ gives
%     \[\bar {\bar \lambda} = R_{r_1, c_1} \diamond \cdots \diamond R_{r_j - 1, c_j -1} \diamond \cdots \diamond R_{c_k, r_k}.\]
%     Since $c_j\ge 2$ is odd, the number of columns remaining in the $j$th rectangle of $\bar {\bar \lambda} $ is even and at least two.
%      \item Otherwise (i.e. if $r_j > 2$ is even, or if $r_j > 3$ is odd), then removing another row of the same rectangle gives 
%     \[\bar {\bar \lambda} = R_{r_1, c_1} \diamond \cdots \diamond R_{r_j - 2, c_j} \diamond \cdots \diamond R_{c_k, r_k}.\]
% \end{enumerate}
In each of these cases, $\bar {\bar \lambda} \in \p$ by induction. 
\end{proof}

It is tempting to imagine that the Sprague-Grundy value (or at least the $\p/\n$ status) of a meld of rectangles does not depend on the order of the rectangles in the meld. 
This is not the case; for example, $\sg([2,2,2] \diamond [4,4]) = 2$ but $\sg([4,4] \diamond [2,2,2]) = 1$ and $[2] \diamond [1] \diamond [2]$ is losing but $[2] \diamond [2]\diamond [1]$ is winning.

We now generalize \cref{thm:ltrl} by showing exactly which rectairs are losing. To this end, we need the following definition.

\begin{definition}
    We say that a rectair $R^k_{r, c}$ is \emph{balanced} if it has even rank, $r$ and $c$ are at least $2$, and when $r$ and $c$ are odd, $k < \min(r, c) -1$. The empty rectair $[ \, ]$ is balanced. 
\end{definition}
The partitions $R_{r,c}, S_k, \ps_k$ and $\dps_k$ are in $\p$ if and only the corresponding rectairs are balanced by \cref{lem:rectangle,prop:pn-of-families}. 

\thmrectairs*
\begin{proof}
    Let $\lambda = R^{k_1}_{r_1, c_1} \diamond \cdots \diamond R^{k_m}_{r_m, c_m}$. We proceed by induction on the size of $\lambda$. The result holds for $\lambda = [ \ ]$. 
    
    For the induction step, suppose without loss of generality that the first player removes the row of $\lambda$ corresponding to row $\ell$ of $\mu = R^{k_p}_{r_p, c_p}$ yielding $\bar \lambda = R^{k_1}_{r_1, c_1} \diamond \cdots \diamond \bar \mu \diamond \cdots \diamond R^{k_m}_{r_m, c_m}.$ 
    We show that the second player has either a parallel or perpendicular winning response on $\bar \lambda$ intersecting $\bar \mu$ yielding a partition $\bar{\bar \lambda}$. If $p < m$, then $R^{k_{p+1}}_{r_{p+1}, c_{p+1}}$ becomes a rectair $\eta = R^{k_{p+1}}_{r_{p+1}, c_{p+1}+t}$ for some nonnegative integer $t$ so that $\bar{\bar \lambda} = R^{k_1}_{r_1, c_1} \diamond \cdots \diamond \bar{\bar \mu} \diamond \eta \diamond \cdots \diamond R^{k_m}_{r_m, c_m}$. 
    To show that $\bar{\bar \lambda}$ is a meld of balanced rectairs, it suffices to show that $\bar{\bar \mu}$ is a meld of balanced rectairs and, if $p < m$, that $\eta$ is a balanced rectair.    
    
First, we treat small values of $k_p$. If $k_p = 0$ or $k_p = 1$, then the second player can play as if on a rectangle \cite{GKM22LCTR}. In particular, consider the subcases $r_p = 2$ and $r_p \geq 3$. 
\begin{itemize}
    \item If $r_p = 2$, then $c_p$ is even and the second player should make a parallel response, yielding $\bar{\bar \mu} = [ \ ]$. If $p < m$, then $\eta = R^{k_{p+1}}_{r_{p+1}, c_{p+1} + c_p}$. Both $\bar{\bar \mu}$ and $\eta$ are balanced, so $\bar{\bar \lambda}$ is a meld of balanced rectairs. 
    \item 
    If $r_p \geq 3$, then $c_p \neq 1$ has the same parity as $r_p$ and the second player should make the perpendicular response, which always yields the rectangle $\bar{\bar \mu} = R_{r_p-1, c_p - 1}$. Also, $\eta = R^{k_{p+1}}_{r_{p+1}, c_{p+1}}$ if $p < m$. 
    Both of these rectairs are balanced, so $\bar{\bar \lambda}$ is a meld of balanced rectairs. 
    %\item If $r_p > 3$, then the second player should make the perpendicular response. Then $\bar{\bar \mu} = R_{r_p-1, c_p-1}$ which is a balanced rectair, and $\bar{\bar \lambda}=R^{k_1}_{r , c_1} \diamond \cdots \diamond \bar{\bar \mu} \diamond \cdots \diamond R^{k_m}_{r_m, c_m}$ is a meld of balanced rectairs. 
\end{itemize}
If $c_p$ or $r_p$ has value $2$ or $3$, then $k_p = 0$ or $k_p = 1$, which we just addressed. Thus, we may assume that $r_p \geq 4$ and $c_p \geq 4$ and $k_p \geq 2$ for the remainder of the proof. 

We now treat large values of $k_p$, i.e., when $k_p = r_p-1$ or $k_p = c_p - 1$. 
In both cases $r_p$ and $c_p$ are even (since $\lambda$ is balanced), and the second player should make the parallel response. 
            %(i.e., row $\ell + (-1)^{\ell+1}$ of $\bar\mu$). 
            Without loss of generality, also suppose $\ell$ is odd so that $\bar{\bar \mu}$ is the result of removing rows $\ell$ and $\ell + 1$ from $\mu = R^{k_{p}}_{r_{p}, c_{p}}$.

% We now consider various cases depending on $p, k_p, r_p,$ and $c_p$, and which row of $R^{k_p}_{r_p, c_p}$ is removed. 
%    \begin{description}
        % \item[$r_p = 2$:] If $r_p = 2$, then $c_p$ is even. Whichever row is removed, the second player should remove the sole row of $\bar \mu$, yielding $\bar{\bar \mu} = [ \ ]$, which is balanced. If $p > 1$, then $\eta = R^{k_{p-1}}_{r_{p-1}, c_{p-1} + c_p}$, which is balanced. 
        % \item[$r_p = 3$:] If $r_p = 3$, then $c_p$ is odd and $k = 0$ or $k = 1$. Whichever row is removed, the second player should remove the last\footnote{Any column works; we choose the last one for specificity.} column of $\bar \mu$, yielding $\bar{\bar \mu} = [c-1, c-1]$. Either way, $\bar{\bar \mu}$ is balanced, as is $\eta = R^{k_{p-1}}_{r_{p-1}, c_{p-1}}$ if $p > 1$. 
 %       \item[$r_p \geq 4$:] 
% and consider the following subcases. 
        \begin{description}
            \item[$k_p = r_p-1$:] We distinguish  three subcases: 
            \begin{itemize}
                \item If $\ell = 1$ then $\bar{\bar \mu} = R^{r_p-3}_{r_p-2, c_p-2}$ and $\eta = R^{k_{p+1}}_{r_{p+1}, c_{p+1}+2}$ if $p < m$, both of which are balanced. 
                \item If $\ell = r_p-1$ then $\bar{\bar \mu} = R^{r_p-3}_{r_p-2, c_p}$ and $\eta = R^{k_{p+1}}_{r_{p+1}, c_{p+1}}$ if $p < m$, both of which are balanced. 
                \item If $3 \leq \ell \leq r_p-3$ then $\bar{\bar \mu} =  R^{r_p - \ell - 2}_{r_p - \ell - 1, c_p - \ell -1} \diamond R^{\ell-2}_{\ell-1, \ell + 1} $ is the meld of balanced rectairs and if $p < m$ then $\eta = R^{k_{p+1}}_{r_{p+1}, c_{p+1}}$ is balanced. 
            \end{itemize}
            \item[$k_p = c_p-1$:] 
            Since we already addressed the case $k_p=r_p-1$, we may assume that $k_p < r_p-1$. 
            In fact, observe that $r_p\ge c_p+2$ and that the first $r_p-k_p\ge 3$ rows are of the same length. 
            %(i.e., row $\ell + (-1)^{\ell+1}$ of $\bar\mu$). 
            Observe that $\ell \neq r_p - c_p$ due to parity. 
            \begin{itemize}
                \item If $\ell < r_p-c_p$ then $\bar{\bar \mu} = R^{k_p}_{r_p-2, c_p}$  which is balanced. 
                \item If $\ell = r_p-1$ then $\bar{\bar \mu} = R^{k_p-2}_{r_p-2, c_p}$ which is balanced. 
                \item If $r_p-c_p < \ell < r_p-1$ then 
$\bar{\bar \mu} = R^{r_p - \ell - 2}_{r_p - \ell - 1, r_p - \ell -1} \diamond R^{k_p - r_p+\ell-1}_{\ell-1, c_p-r_p+\ell + 1}  $ is the meld of balanced rectairs. 
            \end{itemize}
            In each case notice that if $p < m$ then $\eta=R^{k_{p+1}}_{r_{p+1},c_{p+1}}$, which is balanced.

\end{description}
            Having treated the extreme cases $k_p \in \{0, 1, \min\{r_p, c_p\}-1\}$, we turn our attention to the cases when $1 < k_p < \min\{r_p, c_p\}-1$. 
            Recall that $r_p \geq 4$ and $c_p \geq 4$ and $k_p \geq 2$. 
            Regardless of $\ell$, the perpendicular response gives $\bar{\bar \mu} = R^{k_p-1}_{r_p - 1, c_p - 1}$ and $\eta = R^{k_{p+1}}_{r_{p+1}, c_{p+1}}$ if $p < m$, both of which are balanced. 
\end{proof}

\begin{cor}
    A rectair is losing if and only if it is balanced. 
\end{cor}

\begin{proof}
    We show that if $R_{r,c}^k$ is not balanced, then it is winning. 
    Without loss of generality assume $r \le c$. 
    \cref{prop:pn-of-families} settles the case of odd stairs, so assume $r<c$.
    In this case, the removal of the last column yields a balanced rectair.
\end{proof}

\section{Conway-Gurvich-Ho Classification} \label{sec:CG}

In this section, we study the placement of \crim and its restrictions to various types of partitions in the Conway-Gurvich-Ho classification scheme. These games are shown in \cref{fig:cgh} along with some games that have been studied elsewhere.

\begin{figure}[h!]
    \centering
     \includegraphics[scale=.9]{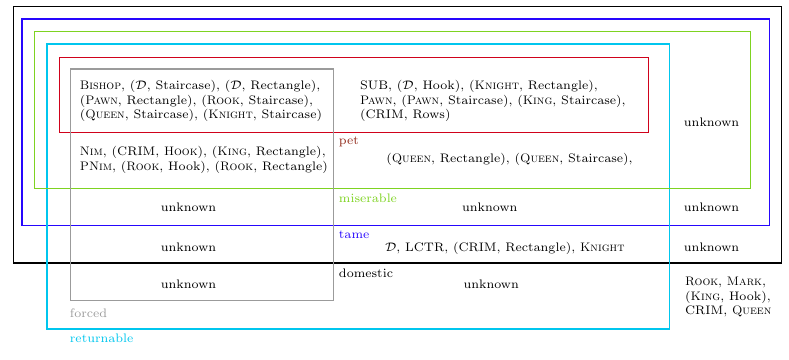}
    \caption{
Conway-Gurvich-Ho Classification of some games. The word ``unknown'' appears in those regions of the Conway-Gurvich-Ho classification scheme in which only contrived combinatorial games are known to appear.
For definitions of the mentioned games we refer the reader to \cite{bouton1901nim,Gottlieb2024LCTR,berlekamp2011puzzles,gottlieb2025impartialchessintegerpartitions}.}
  \label{fig:cgh}
\end{figure}
We begin by considering \crim restricted to the class of row partitions, i.e., those of the form $[c]$ or for nonnegative integers $c$ and $r$. The corresponding results hold for column partitions by conjugate invariance. 

\begin{theorem} \label{thm:cgpairrow}
Let $c$ be a nonnegative integer. Then \[\cp([c]) = \begin{cases}
    (1,0) & \mbox{ if $c$ is odd}\\
    (2,2) & \mbox{ if $c$ is even}.
\end{cases}\] 
\end{theorem}
\begin{proof}
The result holds if $c = 0$ or $c = 1$. Now suppose $c > 1$. There are two kinds of moves that can be made on $[c]$. The first results in the empty partition, with $\cp([\,]) = (0, 1)$. The second results in the partition $[c-1]$. If $c$ is even, then by induction we have $\cp([c-1]) = (1, 0)$, so $\cp([c]) = (2, 2)$. If $c$ is odd, then by induction $\cp([c-1]) = (2, 2)$, so $\cp([c])$ is $(1, 0)$. 
\end{proof}

\cref{thm:rows} shows that $\crim$ on rows is in a region of the Conway-Gurvich classification that is occupied by other games \cite{gottlieb2025impartialchessintegerpartitions}. 

\begin{theorem} \label{thm:rows}
$\crim$ on rows is pet and returnable but not forced.     
\end{theorem}
\begin{proof} 
That $\crim$ on rows is pet follows from \cref{thm:cgpairrow}. $\crim$ on rows is not forced, since $[3] \to [2]$ and $\cp([3]) = (1, 0)$ and $\cp([2]) = (2, 2)$. $\crim$ on rows is returnable, since the only such partitions with Conway pair $(0, 1)$ or $(1, 0)$ are of the form $[2t+1]$ for some positive integer $t$. Every move from $[2t+1]$ to a non-terminal position is a move to $[2t]$, and there is a move from $[2t]$ to $[2t-1]$ which has Conway pair $(1, 0)$. 
\end{proof}

Now we consider $\crim$ on hooks. 

\begin{theorem} \label{thm:cgpairshooks} 
Let $r$ and $c$ be positive integers. Then \[ \cp(\left[ c, 1^{r-1} \right]) = \begin{cases}
(1,0) & \mbox{ if $r = 1$ and $c$ is odd, or if $r$ is odd and $c = 1$ }\\     
(2,2) & \mbox{ if $r = 1$ and $c$ is even, or if $r$ is even and $c = 1$}\\ 
(0,1) & \mbox{ if $r$ and $c$ are both even}\\ 
(3,3) & \mbox{ if $r>1$ and $c>1$ have different parities}\\ 
(0,0) & \mbox{ if $r> 1$ and $c > 1$ are both odd.} 
\end{cases}\]
\end{theorem} 

\begin{proof} 
We proceed by induction on $r+c$. If $r=1$ or $c=1$, the result holds by \cref{thm:cgpairrow}. Observe that $\cp([2, 1]) = (0, 1)$. If $c > 2$, then the partitions reachable in a single move from $[c, 1]$ are $[c]$, $[1]$, $[c-1]$, and $[c-1, 1]$. Thus 
\begin{align*}
    \cp([c, 1]) &= \mex(\cp([c]), \cp([1]), \cp([c-1]), \cp([c-1,1])) \\
    &= \begin{cases}
        \mex((2,2), (1, 0), (1, 0), (3,3)) = (0,1) & \mbox{if $c$ is even} \\
        \mex((1,0), (1, 0), (2,2), (0,1)) = (3, 3) & \mbox{if $c$ is odd.} \\
    \end{cases}
\end{align*}
The result for $[2, 1^{r-1}]$ follows by conjugate invariance. 

If $r > 2$ and $c > 2$, then the partitions reachable in one move from $[c, 1^{r-1}]$ are $[1^{r-1}]$, $[c, 1^{r-2}]$, $[c-1]$, and $[c-1, 1^{r-1}]$. Thus \begin{align*}
    \cp([c, 1^{r-1}]) &= \mex(\cp([1^{r-1}]), \cp([c, 1^{r-2}]), \cp([c-1]), \cp([c-1, 1^{r-1}]))\\
    &= \begin{cases}
        \mex((1,0), (3,3), (1,0), (3,3)) = (0, 1) & \mbox{if $r$ and $c$ are even} \\
        \mex((2,2), (3,3), (2,2), (3,3)) = (0, 0) & \mbox{if $r$ and $c$ are odd}\\
        \mex((1,0), (0,0), (2,2), (0,1)) = (3, 3) & \mbox{if $r$ and $c$ have different parities.} 
    \end{cases}
\end{align*} 
\end{proof}

\cref{thm:hooks} shows that $\crim$ on hooks is in a region of the Conway-Gurvich classification that is occupied by other games \cite{gottlieb2025impartialchessintegerpartitions}. 

\begin{restatable}{theorem}{thmhooks} \label{thm:hooks}
    $\crim$ on hooks is returnable but not forced and miserable but not pet. 
\end{restatable}
\begin{proof}
    Within this proof we refer to $c$ as the \emph{arm} and $r$ as the \emph{leg} of the hook partition $\Gamma_{r, c}$.
    
    $\crim$ on hooks is not forced for the same reason that $\crim$ on rows is not. To see that $\crim$ on hooks is returnable, note that the only positions with Conway pair $(1, 0)$ are rows so this case can be treated as in \cref{thm:rows}. The only positions with Conway pair $(0, 1)$ are hooks with even-length arm and leg. Any move to a row or column partition can be answered with a move to $[]$, which has Conway pair $(0, 1)$. Without loss of generality we need only consider moves $[c, 1^{r-1}] \to [c-1, 1^{r-1}]$ where $r \geq 2$ and $c > 2$.  The position $[c-2, 1^{r-1}]$ has Conway pair $(1, 0)$ and is reachable in one move from $[c-1, 1^{r-1}]$. Thus, $\crim$ on hooks is returnable. 
    
    That $\crim$ on hooks is not pet follows from \cref{thm:cgpairshooks}. To see that $\crim$ on hooks is miserable, we consider partitions with Conway pairs different from $(1, 0)$ and $(0, 1)$ individually. 
    First consider even rows. These admit moves to $[\, ]$ and to an odd row, and these have Conway pairs $(0, 1)$ and $(1, 0)$ respectively. Even rows therefore satisfy the third condition in the definition of miserable. 

    Now suppose $r > 1$ and $c > 1$ have different parities. Let $c > 1$ be odd. From $[c, 1]$ there are moves to $[c]$ and $[c-1,1]$ which have Conway pairs $(1, 0)$ and $(0,1)$ respectively. If $r > 2$ then from $[c, 1^{r-1}]$ there are moves to $[c-1, 1^{r-1}]$ and $[1^{r-2}]$ which have Conway pairs $(0, 1)$ and $(1, 0)$ respectively. Either way, hooks with arms and legs of different parities satisfy the third condition in the definition of miserable.

    Finally, suppose $r>1$ and $c>1$ are both odd. Because the Conway pair of such a partition is $(0, 0)$, there can be a move to a partition with Conway pair neither $(0,1)$ nor $(1, 0)$. Therefore, the second condition in the definition of miserable is satisfied. We conclude that $\crim$ on hooks is miserable. 
\end{proof}

We now consider $\crim$ on rectangles. 

\begin{theorem} \label{thm:cgpairrects}
 Let $r$ and $c$ be positive integers. Then \[\cp([c^r])=\begin{cases}
(1, 0) & \mbox{ if $r = 1$ and $c$ is odd or $r$ is odd and $c=1$} \\
(2, 2) & \mbox{ if $r = 1$ and $c$ is even or $r$ is even and $c = 1$} \\
(0, 0) & \mbox{ if $r > 1$ and $c > 1$ have the same parity} \\
(2, 1) & \mbox{ if $r = 2$ and $c > 1$ is odd or $r > 1$ is odd and $c = 2$} \\ 
(1, 1) & \mbox{ if $r > 2$ and $c > 2$ have different parity}.
\end{cases}\]
\end{theorem}

\begin{proof}
    If $r=1$ or $c=1$, the result holds by \cref{thm:cgpairrow}. Note that $\cp([2,2]) = (0,0)$. If $c > 2$ then the partitions reachable in one move from $[c^2]$ are $[c]$ and $[(c-1)^2]$. Thus \begin{align*}
        \cp([c^2]) &= \mex(\cp([c]), \cp([(c-1)^2])) \\
        &= \begin{cases}
        \mex((2,2), (2,1)) = (0, 0) & \mbox{ if $c$ is even} \\
        \mex((1,0), (0,0)) = (2, 1) & \mbox{ if $c$ is odd.}
    \end{cases} 
    \end{align*}
The result for $(2^r)$ with $r>2$ follows from conjugate invariance. 

    Now suppose $r > 2$ and $c > 2$. The partitions reachable from $[c^r]$ are $[c^{r-1}]$ and $[(c-1)^r]$, so 
    \begin{align*}
        \cp([c^r]) &= \mex(\cp([c^{r-1}]), \cp([(c-1)^r]) \\
        &= \begin{cases}
            \mex((1,1), (1,1)) = (0, 0) & \mbox{ if $r$ and $c$ have the same parity}\\
            \mex((0, 0), (0, 0)) = (1, 1) & \mbox{ if $r$ and $c$ have different parities.} 
        \end{cases}
    \end{align*}
\end{proof}

\cref{thm:rects} shows that $\crim$ on rectangles belongs to a region of the Conway-Gurvich classification that is occupied by other games \cite{gottlieb2025impartialchessintegerpartitions}.

\begin{restatable}{theorem}{thmrects} \label{thm:rects}
    $\crim$ on rectangles is domestic but not tame and returnable but not forced. 
\end{restatable}
\begin{proof}
    It follows immediately from \cref{thm:cgpairrects} that $\crim$ on rectangles is domestic but not tame. It is returnable and not forced by \cref{thm:rows}. 
\end{proof}

\thmCGHmain*
\begin{proof}
To see that $\crim$ is not domestic, observe that $\cp([4, 3]) = (0, 4)$.  To see that CR is not returnable, consider a $(1,0)$-position  $[3, 2^3, 1]$. 
However, after the removal of the last row, the partition $[3, 2^3]$ leads to either $[2^4]$, $[3,2^2]$,$[2,1^3]$, or $[2^3]$, none of which is a $(1,0)$-position.   
\end{proof}

\section{Future Work}  

\label{sec:future} 
Using our results, one can determine the $\p$/$\n$ status for a little over half of all partitions of integers $50$ or smaller. Appendix \ref{sec:app} shows some small losing partitions that our work does not cover.  In this section, we give directions for future work that would enhance our explanatory power. 

% We now state some conjectures concerning $\crim$. 
For partitions of order at most  $50$ and partitions which are (melds of) rectairs, (thick) hooks, OAE-partitions, EAO-partitions, $2$- and $3$-part partitions, 
even rank is a necessary condition for a partition being in $\p$.\footnote{See \cref{thm:2parts,cor:evensareodd,thm:3parts,thm:hooks,thm:oddsareeven,thm:rectairs}.} 
Our main conjecture asserts that this is always true.

\begin{conj}\label{conj:main}
Every $\p$-position is of even rank.
\end{conj}

The above conjecture  not only holds for (thick) hooks and rectangles which are not rows or columns but is also sufficient for these partitions. 
Note that the converse of \cref{conj:main} fails, for example, on odd staircases. 
Our computations suggest that proving \cref{conj:main}  would explain the $\p/\n$-status of roughly half of remaining partitions, not explained by either \cref{thm:oddsareeven,cor:evensareodd,thm:almost-hook,thm:thick-hook,lem:rectangle,prop:hooksquare,thm:3parts,thm:2parts} (see \cref{tab:data}).

\begin{table}[ht]
\centering
\resizebox{\textwidth}{!}{%
\begin{tabular}{c|*{22}{c}}
%\hline
$n$ & $\le$ 15 &  16 & 17 & 18 & 19 & 20 & 21 & 22 & 23 & 24 & 25 & 26 & 27 & 28 & 29 & 30 \\[1mm] \hline
$p(n)$ & 684  & 231 & 297 & 385 & 490 & 627 & 792 & 1002 & 1255 & 1575 & 1958 & 2436 & 3010 & 3718 & 4565 & 5604 \\ %\hline
% $x(n)$ & 98 & 59 & 87 & 115 & 160 & 219 & 293 & 374 & 490 & 636 & 819 & 1022 & 1292 & 1630 & 2038 & 2508 \\ %\hline
% $y(n)$ & 49 & 26 & 49 & 53 & 84 & 103 & 154 & 178 & 251 & 308 & 423 & 497 & 659 & 798 & 1039 & 1230 \\ %\hline
$u(n)$ & 0.14 & 0.26 & 0.29 & 0.30 & 0.33 & 0.35 & 0.37 & 0.37 & 0.39 & 0.40 & 0.42 & 0.42 & 0.43 & 0.44 & 0.45 & 0.45 \\ %\hline
$c(n)$ & 0.07 & 0.11 & 0.16 & 0.14 & 0.17 & 0.16 & 0.19 & 0.18 & 0.20 & 0.20 & 0.22 & 0.20 & 0.22 & 0.21 & 0.23 & 0.22 \\ %\hline
\end{tabular}
}
\vspace{3mm}
\caption{How much do \cref{thm:oddsareeven,cor:evensareodd,thm:almost-hook,thm:thick-hook,lem:rectangle,prop:hooksquare,thm:3parts,thm:2parts} explain? The first row represents $n$, the size of the partition. The second row is the number $p(n)$ of partitions of size $n$. The third row is the fraction $u(n)$ of partitions of size $n$ whose $\p/\n$-status is not explained by our results. The fourth row is the fraction $c(n)$ of partitions of size $n$ whose $\p/\n$-status is explained by our results, assuming \cref{conj:main} holds.
}
\label{tab:data}
\end{table}

\vspace{-13mm}

\begin{obs}
Suppose \cref{conj:main} holds, and let $\lambda$ be an $\n$ position of even rank. 
Then difference between the first two rows or the first two columns is odd.
Furthermore, the winning response is either the removal or first row or first column.
\end{obs}
We propose various extensions of the results obtained in \cref{sec:CR}. 
%\subsubsection*{Partitions with small number of parts}
%
%\subsection{Partitions with small number of parts}
In \cref{thm:3parts} we characterized all losing positions for any partition consisting of three parts.
\begin{question}
    Which partitions consisting of four parts are losing?
\end{question}

\begin{question}
    Find Sprague-Grundy values for partitions with few parts. 
\end{question}

\medskip
Concerning rectairs, in \cref{prop:pn-of-families}, we showed that $S_n \in \p$ if and only if $n$ is even, that $\ps_n \in \n$, and that $\dps_n \in \p$. 
We now offer two conjectures that generalize and extend these results by postulating the Sprague-Grundy values of $R^k_{r, r}$ and $R^k_{r, r+1}$. 

\begin{conj} Let $r$ and $k$ be integers with $0 \leq k < r$. Then
    \[
    \sg(R_{r, r}^k)=
    \begin{cases}
        0& r \text{ is even or }k<r-1,\\
        1& r\in \{3,5\}, \, r \text{ is odd, and } k = r-1,\\
        2& \text{otherwise.}
    \end{cases}    
    \]
\end{conj}

\begin{conj} For $r \geq 7$ we have \[\sg(R_{r,r-1}^k)=
    \begin{cases}
    3& k=r-2 \text{ and } r \text{ odd},\\
    1& \text{otherwise.}
    \end{cases}\]
Also, $\sg(R^1_{3, 2}) = \sg(R^3_{5, 4}) = 3$ and $\sg(R^2_{4, 3}) = 2$ and $\sg(R^4_{6, 5} )= 5$. All other rectairs have Sprague-Grundy value $1$. 
\end{conj}

We conjecture that the Sprague-Grundy values of the families $\s_n$ and $\ps_n$ of rectairs are as follows. 
\begin{conj} 
$SG(\s_n) = 0$ if $n$ is even. Also, $SG(\s_n) = 1$ if $n = 1, 3, 5$ and $\mathcal \s_n = 2$ if $n$ is odd and 7 or larger. 
\end{conj}

\begin{conj}
    $SG(\mathcal{\ps}_n) = 2$ if $n = 1, 3$ and $SG(\mathcal{\ps}_n) = 1$ if $n \geq 6$ is even and $SG(\mathcal{\ps}_5) = 5$ and $SG(\mathcal{\ps}_n) = 3$ otherwise. 
\end{conj}

\begin{question}
Find the Sprague-Grundy values of all rectairs. 
\end{question}

\bigskip
\noindent In \cref{lem:hook} we gave a formula for the Sprague-Grundy values of hooks. 
\begin{conj} \label{conj:hookswap}
    A partition $\lambda$ has Conway pair $(0, 1)$ if and only if $\lambda$ is an even-by-even hook. 
\end{conj} 
\noindent Positions with Conway pair $(0,1)$ or $(1,0)$ play an important role in the Conway-Gurvich-Ho classification scheme. Thus, \cref{conj:hookswap} motivates the following question. 
\begin{question}    
    Characterize positions with Conway pair $(1, 0)$. 
\end{question}

\subsection{Understanding operations on partitions}

\cref{thm:rectairs,thm:ltrl} show that 
the meld of certain losing positions remains in $\p$. 
We believe similar is true for thick hooks.
\begin{conj}
    A meld of losing thick hooks is losing.
\end{conj}

Can we obtain other $\p$ positions via simple operations on partitions which we know to be in $\p$? 
To this end, we will introduce two new operations.  

\paragraph{Padding partitions:} For $\lambda = [\lambda_1, \ldots, \lambda_r]$ let $\hat \lambda = [\lambda_1, \lambda_1, \lambda_1, \lambda_2, \ldots, \lambda_r]$. If $\lambda \in \p$, one might believe that $\hat \lambda \in \p$.
In particular, it holds if $\lambda$ is an $\OAE$-partition, an $\EAO$-partition, or a rectair. 
It does not hold in general, since $[4,1]\in \p$ but $[4,4,4,1] \in \n$. The converse also fails since $[5^5, 3^7, 1^7] \in \n$ while $[7^5, 5^7, 3^7] \in \p$.

\begin{question} \label{question:q1}
    Characterize losing partitions $\lambda$ for which $\hat \lambda$ is losing.
\end{question}
\begin{question} \label{question:q2}
    Characterize partitions $\lambda$ for which $\hat \lambda$ losing implies $\lambda$ is losing. 
\end{question}

It similarly misleading to think that if $\lambda \in \p$, then $[\lambda_1, \ldots, \lambda_r, 1, 1] \in \p$. 
In particular, it is true if $\lambda \in \EAO$. 
Again, the implication does not hold in general; $\lambda = [4,3]$ is a counterexample. 
The converse also fails since $[3,1,1] \in \p$ and $[3]\in \n$. 
Can we answer questions in the spirit of 
\cref{question:q1,question:q2} in this context? 
\paragraph{Stacking partitions:}
The \emph{Durfee length} of a partition $\lambda$ is the greatest integer $d$ such that $d \leq \lambda_d$. We denote by $d(\lambda)$ the Durfee length of $\lambda$. 
    Let $\lambda = [\lambda_1,\ldots, \lambda_r]$ and $\mu = [\mu_1, \ldots, \mu_k]$ be partitions such that $\lambda_{d+i}=d$ for all $1\le i\le k$ and 
    $\lambda_d\ge d+\mu_1$, where $d=d(\lambda)$.
    We define their \emph{stack} as
    \[
    \lambda\lhd \mu=[\lambda_1,\dots,\lambda_d,\lambda_{d+1}+\mu_1, \dots,\lambda_{d+k}+\mu_k,\lambda_{d+k+1},\dots,\lambda_r].
    \]   
   If $\lambda\lhd \mu$ is defined, then $\mu\lhd\lambda$ is not. Furthermore, the stack operation is associative and can be extended to multiple partitions.
    Thus, the stack of several partitions is defined for at most one ordering of the partitions.
    See \cref{fig:hook-stack} for an example.

\begin{conj}
\label{conj:bloby}
    Let $\lambda$ and $\Gamma_{r,c}$ be losing and suppose $\lambda^*=\Gamma_{r,c}\lhd\lambda$ is defined. Then $\lambda^*\in \p$.
\end{conj}

 \begin{figure}
     \centering
     \includegraphics[scale=.75]{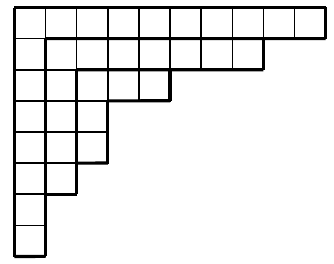}
     \caption{The partition $\Gamma_{8, 10} \lhd \Gamma_{5,7} \lhd \Gamma_{3,3}$.}
     \label{fig:hook-stack}
 \end{figure}

If the above conjecture is true, \cref{prop:hooksquare} follows immediately as a consequence, as does its generalization below.

\begin{obs}
    Assume \cref{conj:bloby}.
    Then a stack of odd hooks, with last hook being $[1]$, is winning. Removing the top row on any such partition is a winning move. 
\end{obs}

\begin{proof} 
    Let $\lambda = \Gamma_{2m_1+1, 2m'_1 + 1} \lhd \Gamma_{2m_2+1, 2m'_2+1} \lhd \cdots \lhd \Gamma_{2m_k+1, 2m'_k+1} \lhd [1]$. 
    Let $\bar \lambda$ be the partition that results from removing the top row of $\lambda$. Then 
\begin{align*}
\bar \lambda &= \Gamma_{2m_1, 2m'_2+2} \lhd \Gamma_{2m_2, 2m'_3+2} \lhd\dots\lhd \Gamma_{2m_k, 2}
\end{align*}
    Because each hook in the stack of $\bar \lambda$ has an even number of rows and an even number of columns, it follows from \Cref{conj:bloby} that $\bar \lambda$ is losing and $\lambda$ is winning. \qedhere
\end{proof}

Before we can state another conjecture related to the stack of partitions, we need to introduce more notation.
We say that a partition $\lambda$ is \emph{parity self-conjugate} if $\lambda_k \equiv \lambda_k' \bmod 2$ for every $k$. Note that parity self-conjugate is a common generalization of self-conjugate and quadrated. 

\begin{conj}
\label{conj:pself}
    A parity self-conjugate partition is in $\n$ if and only if it is the stack of odd self-conjugate hooks and $[1]$. 
\end{conj}

We verified \Cref{conj:pself} for partitions of $n \le 50$.

\backmatter
\vfill
\section*{Declarations}

\begin{itemize}
\item \textbf{Funding:}
The second author was supported by an internal grant from Rhodes College.
The third author was supported in part by the Slovenian Research and Innovation Agency (ARIS), research program P1-0383 and research project N1-0370.
\item \textbf{Conflict of interest:}
None
\item \textbf{Ethics approval and consent to participate:}
All authors consent
\item \textbf{Consent for publication:}
All authors consent
\item \textbf{Data availability:} 
Not applicable
\item \textbf{Materials availability:}
Not applicable
\item \textbf{Code availability:} 
Not applicable
\item \textbf{Author contribution:}
All authors contributed equally towards preparation of this manuscript.
\end{itemize} 

\newpage
\bibliography{bibliography}% common bib file

\begin{appendices}
\section{ Some Unexplained Losing Partitions}\label{sec:app}

\centering

\renewcommand{\arraystretch}{2.5} 
\begin{tabular}{m{3cm}m{3cm}m{3cm}m{3cm}}
%1
    \drawpartition{4, 2, 2, 2}
    & \drawpartition{6, 2, 2, 2}
    & \drawpartition{6, 4, 1, 1}
    & \drawpartition{6, 3, 2, 2}\\[5mm]
% row 2
    \drawpartition{6, 4, 2, 1}
    & \drawpartition{4, 4, 3, 3}
    & \drawpartition{6, 4, 2, 2}
    & \drawpartition{6, 6, 1, 1} \\[5mm]
% row 3
    \drawpartition{8, 2, 2, 2}
    & \drawpartition{8, 4, 1, 1}
    & \drawpartition{6, 4, 3, 2}
    & \drawpartition{6, 5, 3, 1} \\[6mm]
%row 4   
    \drawpartition{5, 4, 2, 2, 1}
    & \drawpartition{5, 3, 3, 2, 2}             
    & \drawpartition{5, 3, 2, 2, 2}                 
    & \drawpartition{5, 3, 3, 2, 1} \\[7mm]
% row 5 
    \drawpartition{8, 3, 2, 2} 
    & \drawpartition{8, 4, 2, 1}     
    & \drawpartition{8, 6, 1, 1}
    & \drawpartition{10, 2, 2, 2} \\[5mm]
    
 % row 6   
    \drawpartition{6, 4, 3, 3}
    & \drawpartition{6, 4, 4, 2}
    & \drawpartition{8, 4, 2, 2}
    & \drawpartition{8, 5, 2, 1} \\[6mm]
    
 % row 7   
    \drawpartition{7, 3, 2, 2, 2} 
    & \drawpartition{7, 3, 3, 2, 1} 
    & \drawpartition{7, 4, 2, 2, 1} 
    & \drawpartition{7, 4, 3, 1, 1} \\[8mm]

    % row 8
      \drawpartition{6, 3, 2, 2, 1, 1} 
    & \drawpartition{6, 2, 2, 2, 2, 2} 
    & \drawpartition{6, 3, 2, 2, 2, 1} 
    & \drawpartition{6, 4, 2, 2, 1, 1} \\[10mm]
    
% row 9   
   \drawpartition{6, 2, 2, 2, 1, 1}   
    & \drawpartition{8, 2, 2, 2, 1, 1} 
    & \drawpartition{8, 4, 1, 1, 1, 1} 
    &\\[10mm]

% row 10 
    \drawpartition{7, 5, 2, 1, 1}
    & \drawpartition{5, 5, 2, 2, 2} 
    & \drawpartition{10, 4, 1, 1} 
    & \\[5mm]

   \end{tabular}




\end{appendices}


\end{document}